\documentclass{amsart}
\usepackage{amsmath,amsthm,amssymb,latexsym,ifthen,graphicx}
\usepackage{mathrsfs}
\usepackage{color}
\usepackage[colorlinks,citecolor=blue,urlcolor=black, linkcolor=blue, backref]{hyperref}

\usepackage{mathrsfs}
\newtheorem{theorem}{Theorem}[section]

\newtheorem{corollary}[theorem]{Corollary}

\newtheorem{definition}[theorem]{Definition}

\newtheorem{First Proof}[theorem]{First Proof}
\newtheorem{fact}[theorem]{Fact}
\newtheorem{Second Proof}[theorem]{Second Proof}

\newtheorem{lemma}[theorem]{Lemma}

\newtheorem{proposition}[theorem]{Proposition}
\newtheorem{question}{Question}
\newtheorem{remark}[theorem]{Remark}

\begin{document}


\title{When is $A + x A =\mathbb{R}$}

\author{Jinhe Ye}
\address{Mathematical institute\\
  University of Oxford\\
  Oxford, Oxford ox2 6gg, UK}
\email{jinhe.ye@maths.ox.ac.uk}
\author{Liang Yu}
\address{School of mathematics\\
  Nanjing University\\
  Nanjing, Jiangsu 210093, People's Republic of China}
\email{yuliang.nju@gmail.com}
\author{Xuanheng Zhao}
\address{School of mathematics\\
  Nanjing University\\
  Nanjing, Jiangsu 210093, People's Republic of China}
\email{xuanheng21@gmail.com}


\begin{abstract}
For a subgroup $A$ of $(\mathbb{R},+)$ and a real $x$, define $A+xA=\{ a + x b : a, b \in A \}$ and $X_A=\{x\in \mathbb{R}:A + x A=\mathbb{R}\}$. We show that there is an $F_\sigma$ subgroup $A$ of $(\mathbb{R},+)$ such that $\mathrm{dim_H}  (A) \le \frac{1}{2}$ and $X_A \neq \emptyset$. However, if $A \subseteq \mathbb{R}$ is a subring of $\mathbb{R}$ and $X_A \neq \emptyset$, then $A =\mathbb{R}$. Moreover,
  assuming $\mathrm{CH}$ (the continuum hypothesis), there is a subgroup $A$ of $(\mathbb{R},+)$ with $\mathrm{dim_H}  (A) = 0$
  such that $X_A =\mathbb{R}\backslash \mathbb{Q}$.
 The proof of this theorem combines several techniques in recursion theory and algorithmic dimension. Several other theorems on analytic subgroups and subfields of the reals are presented. We also discuss some of these results in the $p$-adics.
\end{abstract}

\maketitle

\section{Introduction}
In this paper we prove several results concerning the ``sizes" of subgroups and subrings of the reals. Here by ``size", we typically refer to the Hausdorff measure and dimension of such objects. The story of ``sizes" has a rich history, which begins with a classical result in real analysis.

\begin{theorem}[Steinhaus \cite{Steinhaus1920}] Suppose $A \subseteq \mathbb{R}^n$ is Lebesgue measurable and has positive measure. Then the difference set $A-A:=\{x-y:x,y \in A\}$ contains a ball with positive radius whose center is at the origin.
\end{theorem}

The next corollary follows immediately.

\begin{corollary}\label{steinhaus}
If $A$ is a measurable proper subgroup of $(\mathbb{R},+)$, then $A$ is null.\footnote{We call a set $A \subseteq \mathbb{R}$ \textit{null} if it has Lebesgue measure zero.}
\end{corollary}

Subsequently, Volkmann and Erd{\H o}s initiated the study of the dimension of subgroups/rings of the reals in the 1960s. In {\cite{1}} they showed that for each $\alpha \in [0,1]$, there is a Borel subgroup of
$(\mathbb{R},+)$ with Hausdorff dimension $\alpha$. Edgar and Miller {\cite{MR1948103}}, and independently Bourgain {\cite{MR1982147}} showed
that an \textit{analytic} (see the definition before Fact \ref{Jech}) subring of $\mathbb{R}$ either has Hausdorff dimension 0 or
is all of $\mathbb{R}$. Mauldin {\cite{MR3505676}} showed that
assuming $\mathrm{CH}$ (the continuum hypothesis), for each $\alpha$, $0 \leqslant \alpha \leqslant 1$, there is a subfield of $\mathbb{R}$ with Hausdorff dimension $\alpha$. Following the strategy of ``discretization'' used by Bourgain in {\cite{MR1982147}}, de Saxc\'e \cite{MR3606726} considered the problem in the setting of a connected simple real Lie group $G$ endowed with a Riemannian metric and showed that there is no Borel measurable dense subgroup of $G$ with Hausdorff dimension strictly between 0 and $\mathrm{dim_H}  (G)$. For a more detailed discussion on the early history, see {\cite{MR1948103}} p.1122.  The subgroups and subrings of the reals are also natural objects appearing in geometric measure theory, see \cite[Section 12.4]{MR3236784} for example.

In this paper we consider the following related question.

\begin{definition}
    Given a subgroup $A$ of $(\mathbb{R},+)$, define $$X_A=\{x\in \mathbb{R}:A+xA=\mathbb{R}\}.$$
\end{definition}

\begin{question}\label{q1}
  Suppose $A \subseteq (\mathbb{R},+,\cdot)$ is a subobject in some algebraic sense e.g. subgroup, subring or subfield. Is there an $x \in \mathbb{R}$ such that
 $$ A + x A := \{ a + x b : a, b \in A \} =\mathbb{R}? $$
  And if such an $x$ exists, what do we know about the size of $A$? What do we know about $X_A$?
\end{question}

We prove the following main results.

\begin{theorem}
  \label{2}
 There is an $F_\sigma$ subgroup $A$ of $(\mathbb{R},+)$ such that $\mathrm{dim_H}(A)\le\frac{1}{2}$ and $X_A \neq \emptyset$. 
\end{theorem}

\begin{theorem}\label{main}
    Assume $\mathrm{CH}$, holds:
  $2^{\aleph_0} = \aleph_1$. There is a group $A$ with $\mathrm{dim_H}  (A) = 0$
  such that 
 $X_A=\mathbb{R}\backslash \mathbb{Q}$.
\end{theorem}

Theorem \ref{2} is proved via a concrete construction of $A$ and $x \in X_A$. For Theorem \ref{main}, we construct $A$ using the concept of genericity and algorithmic dimension. We believe this method may be used to construct other exotic subsets of reals.

 We also consider what would happen if we restrict Question \ref{q1} to the analytic subgroups. By the Marstrand projection theorem,  if $A$ is an analytic subgroup of the reals and $\mathrm{dim_H}(A)>\frac{1}{2}$, then $X_A$ is co-null\footnote{We call a set $A \subseteq \mathbb{R}$ \textit{co-null} if its complement is null.} (Proposition \ref{36}). Our main results improve this proposition in terms of dimension. The restriction on analytic sets also enables us to use descriptive set theory. For example, if $A$ is an analytic subgroup of the reals and $x \in X_A$, then there is an $F_\sigma$ subgroup $B\subseteq A$ such that $x \in X_B$ (Theorem \ref{fsigma}).

For subrings, the condition is different. If $A$ is a subring of $\mathbb{R}$ and $X_A \neq \emptyset$, then $A =\mathbb{R}$ (Proposition \ref{1}). 
  
Besides Question \ref{q1}, we also study the maximal subfields of $\mathbb{R}$ without a given point because the proof techniques suggest so. Given a field $K$ and $a \in K$, a \textit{maximal subfield of $K$ without $a$} is a maximal (with respect to inclusion) element of $\{L \subseteq K: L \text{ is a subfield of } K \text{ and } a \notin L\}$. Such subfields have been well investigated in the literature. For example, Quigley \cite{MR137705} described the structure of $K / M$ in detail where $K$ is an algebraically closed field and $M$ is a maximal subfield of $K$ without a given point $a$, and used Galois theory to give existence proofs of $M$ more precise than that trivially given by Zorn's Lemma. We restrict our attention to the case that $K=\mathbb{R}$. We show that the usage of $\mathrm{AC}$ (the Axiom of Choice) is necessary for the existence of a maximal subfield of $\mathbb{R}$ without a given point (Corollary \ref{26}). On the other hand, by assuming CH, we construct a maximal subfield of $\mathbb{R}$ with Hausdorff dimension 0 such that some given point is not in its algebraic closure relative to $\mathbb{R}$ (Proposition \ref{47}).

The structure of the paper is as follows: Section \ref{s2} deals with subrings and subfields of $\mathbb{R}$. In Section \ref{s3} we discuss some simple results concerning the basic properties of subgroups and prove Theorem \ref{2}. In Section \ref{s4} we assume $\mathrm{CH}$ in the constructions and prove Theorem \ref{main}. We also show that $A$ in Theorem \ref{main} cannot be $F_{\sigma}$ (Proposition \ref{48}). Due to the variety of the tools used for the results and the independence of the methods, we postpone the introduction of notation and terminologies, and only recall them when being used.

\section{Subrings and Subfields}\label{s2}

Recall that an ordered
field $(R,+,\cdot,0,1,\le)$ ($\le$ is a total order on $F$; if $a \le b$ then $a+c\le b+c$; and if $0 \le a$ and $0 \le b$ then $0 \le a \cdot b$)  is \textit{real closed} if:

 (i) Any positive element has a square root in $R$, and

 (ii) Any polynomial equation $f  (x) = 0$ where $f  (x) \in R [x]$ is of odd
degree has a root in $R$.

\begin{fact}[Artin-Schreier, see {\cite[p.674]{MR1009787}}]
  \label{3}Let $C$ be an algebrically closed field and $F$ be a proper subfield of
  $C$ such that $C/F$ is finite. Then $F$ is real closed and
  $C = F \left ( \sqrt{- 1} \right)$.
\end{fact}

\begin{fact}[``Lying-over'', see {\cite[p.411]{MR1009787}}]
  \label{4}Let $E$ be a commutative ring, $R$ a subring such that $E$ is
  integral over $R$ (for every element $\alpha$ of $E$, there is a monic polynomial $f(x) \in R[x]$ such that $f(\alpha)=0$). Then for each prime ideal $p$ of $R$, there is a prime ideal $P$ of $E$ such that $p = P \cap R$. In particular, if $E$ is a field, so is $R$.
\end{fact}

Let $p$ be a prime number. Let $a=\frac{b}{c}, b, c \in \mathbb{Z}$ be nonzero. We extract from $b$ and $c$ as high a power of the prime number $p$ as possible, namely choose $m$ an interger such that
$$a=p^m\frac{b'}{c'}, (b'c',p)=1,$$
and we put the \textit{$p$-adic absolute value} $|a|_p=\frac{1}{p^m}$. The field of \textit{$p$-adic numbers} $\mathbb{Q}_p$ is the completion of $\mathbb{Q}$ with respect to the metric $d_p(x,y)=|x-y|_p$.

The following lemma is presumably well-known.

\begin{lemma}\label{lem:finite}
(i) If $F$ is a subfield of $\mathbb{R}$ such that $\mathbb{R}/F$ is finite, then $F=\mathbb{R}$.

(ii) If $F$ is a subfield of $\mathbb{Q}_p$ such that $\mathbb{Q}_p/F$ is finite, then $F=\mathbb{Q}_p$.
\end{lemma}

\begin{proof}
(i) follows from Fact~\ref{3}. For (ii), consider the normal closure $K$ of $\mathbb{Q}_p/F$; $K$ is finite over $\mathbb{Q}_p$ and hence any automorphism of $K$ is continuous. Hence any automorphism of $K$ restricts to the identity on $\mathbb{Q}_p$, which implies that $\mathbb{Q}_p=F=K$.
\end{proof}

Recall that we define $X_A=\{x \in \mathbb{R}:A+xA=\mathbb{R}\}$.

\begin{proposition}
  \label{1} (i) If $A$ is a subring of $\mathbb{R}$ and $X_A \neq \emptyset$, then $A =\mathbb{R}$.

  (ii) If $A$ is a subring of $\mathbb{Q}_p$ and there
  is $x$ such that $A + x A =\mathbb{Q}_p$, then $A =\mathbb{Q}_p$.
\end{proposition}

\begin{proof}
  (i). By our assumption, $\mathbb{R}$ is finite over $A$
  and hence by Fact \ref{4}, $A$ is a field since $\mathbb{R}$ is a field.
  Then Lemma \ref{lem:finite} (i) implies $A =\mathbb{R}$.
  
  (ii).  By Fact~\ref{4} and our assumption, $A$ is a subfield of $\mathbb{Q}_p$. If $A\neq \mathbb{Q}_p$, then $[\mathbb{Q}_p:A]\neq 1$ is finite. Then Lemma \ref{lem:finite} (ii) implies $A =\mathbb{Q}_p$.
\end{proof}

Both $\mathbb{Q}_p$ and $\mathbb{R}$ are \textit{Polish} (separable completely metrizable) spaces. A subset $A$ of a Polish space $X$ is \textit{analytic} if it is the projection of a Borel set $B\subseteq X \times X$.

\begin{fact} [see {\cite[Theorem 11.18]{MR1940513}} ]\label{Jech}
Every analytic subset of $\mathbb{R}$ is measurable.
\end{fact}
The proof of this fact is general enough to work for other measures in Polish spaces. For example, in the following Proposition \ref{25}, we will use the fact that every analytic subset of $\mathbb{Q}_p$ is $\mu$-measurable where $\mu$ is the Haar measure on $\mathbb{Q}_p$.

Given a field $K$ and $a \in K$, a \textit{maximal subfield of $K$ without $a$} is a maximal (with respect to inclusion) element of $\{L \subseteq K: L \text{ is a subfield of } K \text{ and } a \notin L\}$.

\begin{proposition}\label{25}
Let $K=\mathbb{Q}_p$ or $\mathbb{R}$. Let $F$ be an analytic subfield of $K$ with some $x\in K\setminus F$. Then there is $y\in K\setminus F$ such that $x\notin F(y)$. Hence $F$ is not a maximal subfield of $K$ without $x$.
\end{proposition}
\begin{proof}
    
     Assume otherwise, then $F$ is a maximal subfield of $K$ avoiding $x$. This implies that $K/F$ is algebraic. Indeed, if $x$ is not algebraic over $F$, then $x \notin F(x^2)$. 
     Now if $K/F$ is not algebraic, then $K$ contains a copy of $F(X)$ for $X$ transcendental over $F$,  which does not contain $x$. 

     So each $y\in K\setminus F$ satisfies a polynomial over $F$. Let $D_n$ denote the set of elements in $K$ whose minimal polynomial over $F$ is of degree at most $n$, this is an analytic set and $\bigcup_{n\in \mathbb{N}}D_n=K$. In particular, $D_n$ is measurable by Fact~\ref{Jech}. It is clear that $D_n$ is closed under multiplication by $-1$. Note further that the difference of 2 elements of degree at most $n$ over $F$ has degree at most $n^2$. Moreover, there must be $n$ such that $\mu(D_n)>0$ where $\mu$ denote the Haar measure on $K$. So $D_{n^2}=K$ by the Steinhaus theorem for Haar measure on locally compact groups (see \cite{MR308368}). This means that any element in $K/F$ is of degree at most $n$ for some $n\in\mathbb{N}$.  By the primitive element theorem, we have that $[K:F]$ is at most $n$. Then by Lemma~\ref{lem:finite}, $K=F$, a contradiction.
\end{proof}

The assumption that $F$ is analytic is used to guarantee that $D_n$ is measurable, so that we can use the Steinhaus theorem. If $D_n$'s are measurable a priori, then the proof will go through as is.
It follows from $\mathrm{AC}$ that there is a non-measurable set. Under some weaker set theoretic axioms, the situation will be different. We use $\mathrm{DC}$ to denote \textit{the axiom of dependent choice}, that is: suppose $R$ is a relation on a nonempty set $X$; if for every $a \in X$, there is an $b \in X$ such that $a R b$, then there is a sequence $\{ x_n\}_{n \in \omega}$ in $X$ such that $x_n R  x_{n+1}$ for all $n \in \omega$. $\mathrm{DC}$ follows from $\mathrm{AC}$ trivially. Solovay \cite{MR0265151} proved that $\mathrm{ZF}+\mathrm{DC}+\mbox{\text{Every subset of $\mathbb{R}$ is measurable}}$ is consistent.

By the above discussion and the proof of Proposition \ref{25}, we have the following conclusion: without $\mathrm{AC}$, one cannot construct a maximal subfield of $\mathbb{R}$ without a given point.

\begin{corollary}\label{26}
Assume $\mathrm{ZF}+\mathrm{DC}+\textup{Every subset of $\mathbb{R}$ is measurable}$. For any subfield $F$ of $\mathbb{R}$ and $x\in \mathbb{R}\setminus F$, $F$ is not a maximal subfield of $K$ without $x$.
\end{corollary}

 \section{Subgroups}\label{s3}

We recall some terminology and notation from {\cite{MR3236784}}. Let $A \subseteq
X$, where $(X,d_X)$ is a metric space. Let the diameter of $A$, written $\left| A \right|$, be the supremum of
the distances between any two points in $A$, i.e. $\left| A \right| = \sup_{x, y \in A} d_X(x,y)$. Suppose $s \ge 0$.
For each ${\ensuremath{\delta}}> 0$, we define the \textit{$s$-dimensional Hausdorff measure} of $A$ by
\[ \mathcal{H}^s  (A) = \lim_{\delta \rightarrow 0}\mathcal{H}^s_{\delta}  (A) = \lim_{\delta \rightarrow 0}\inf \left\{ \sum_{i = 1}^{\infty} \left| U_i \right|^s :
   \{ U_i \} \mbox{ is } \text{a}\mbox{ cover} \mbox{ of } A
   , 0 < \left|U_i\right| \leqslant \delta, \forall i\right\} . \]
 $\mathcal{H}^s  (A)$ can be  (and usually is) 0 or $\infty$.
There is a
critical value of $s$ at which $\mathcal{H}^s  (A)$ ``jumps'' from ${\infty}$ to 0.
This critical value is called the \textit{Hausdorff dimension} of $A$, written
$\mathrm{dim_H}  (A)$. In other words,
\[
\mathrm{dim_H}(A)=\inf\{s:\mathcal{H}^s(A)=0\}=\sup\{s: \mathcal{H}^s(A)=\infty\}.
\]

We gather some of the basic properties of subgroups of the reals. Recall that we define $X_A=\{x \in \mathbb{R}:A+xA=\mathbb{R}\}$.

\begin{proposition} 
  \label{5} Suppose that $A$ is a subgroup of $\mathbb{R}$.

   (i) For all $x \in \mathbb{R}$, $A+xA$ is a subgroup of $\mathbb{R}$.
   
   (ii) If $p \in \mathbb{Q}$ and $p \in X_A$, then $A =\mathbb{R}$.
  
   (iii) If $k \in \mathbb{N}$, $p_0, \ldots, p_k \in \mathbb{Q}$ and $A +
  \sum_{i = 0}^k p_i A = \mathbb{R}$, then $A =\mathbb{R}$.
  
   (iv) Define the tensor product \[\mathbb{Q}\otimes_{\mathbb{Z}} A=\left\{ \sum_{i = 0}^k p_i a_i : k \in
  \mathbb{N}, p_i \in \mathbb{Q}, a_i \in A \mbox{ for } i = 0, \ldots, k
  \right\}.\]
  Then $\mathrm{dim_H} \left ( \mathbb{Q}\otimes_{\mathbb{Z}} A \right) = \mathrm{dim_H}  (A)$.

  (v) If $x \in X_A$, then for all $m \in \mathbb{Z}-\{0\}$, $mx\in X_A$.

  (vi) If $x \in X_A$, then for all $m, n \in \mathbb{Z}-\{0\}$, $x+ \frac{n}{m}\in X_A$.
\end{proposition}

\begin{proof}
(i). By the definition of a subgroup.

   (ii). Suppose that $p = \frac{m}{n}$ for some integers $m$ and $n \neq 0$.
  Then for any $c \in \mathbb{R}$, there are $a, b \in A$ such that $a +
  \frac{m b}{n} = c$. So $n c = n a + m b \in A$. Thus $\mathbb{R}= \frac{1}{n} A$. Since $\mathrm{Char}(\mathbb{R})=0$, we have $\mathbb{R}=n\mathbb{R}=A$.
  
   (iii). Suppose that $p_i = \frac{m_i}{n_i}$ for some integers $m_i$ and $n_i
  \neq 0$. Let $n$ be the least common multiple of those $n_i$. Then $\mathbb{R}=A +
  \sum_{i = 0}^k p_i A \subseteq \frac{1}{n}A$, so $A =\mathbb{R}$.
  
   (iv). Note that $\mathbb{Q}\otimes_\mathbb{Z}A \subseteq \bigcup_{n \ge 1}\frac{1}{n}A$.
  But it is clear that $\mathrm{dim_H}  (A) = \mathrm{dim_H}  (\frac{1}{n} A)$ for
  any $n \ge 1$ since $f:\mathbb{R} \to \mathbb{R},x\mapsto nx$ is bi-Lipschitz. So
  \[ \mathrm{dim_H}  (A) \leqslant \mathrm{dim_H} \left ( \mathbb{Q}\otimes_\mathbb{Z}A \right) \leqslant \mathrm{dim_H} \left ( \bigcup_{n \ge 1} \frac{1}{n} A
    \right) = \sup_i \mathrm{dim_H}(\frac{1}{n} A) = \mathrm{dim_H}  (A).\]

    (v). Note that $$A+mxA=m(\frac{A}{m}+x A)\supseteq m(A+xA)=m \mathbb{R}.$$
  Since $\mathrm{Char}(\mathbb{R})=0$, we have $A+mxA= \mathbb{R}$.

   (vi). Note that
   
   $$A+(x+\frac{n}{m})A=A+(m x+n) \frac{A}{m} \supseteq A+(mx+n)A = A+mxA.$$
   By (v), we have $A+(x+\frac{n}{m})A= \mathbb{R}$.
\end{proof}
\begin{remark}
    Let $X$ be a Hamel basis of $\mathbb{R}$ over 
$\mathbb{Q}$ with $\mathrm{dim_H}  (X)=0$  (such a basis exists by Lutz, Qi, and Yu \cite{MR4766390}), and let $A$ be the group generated by $X$. then by Proposition
\ref{5}(iv),
\[ \mathrm{dim_H}  (A) = \mathrm{dim_H}  (\mathbb{R}) = 1 > \mathrm{dim_H}  (X). \]
Hence the additive group generated by a basis $X$ may have Hausdorff dimension greater than $\mathrm{dim_H}  (X)$.
\end{remark}

We can easily deduce from the Marstrand projection theorem: if $A$ is ``regular'' and ``large'' enough, then there is $x$ such that $A+xA=\mathbb{R}$.

 \begin{fact}[Marstrand \cite{MR0063439}] \label{fact_marstrand}
     Let $F \subseteq \mathbb{R}^2$ be an analytic set with $\mathrm{dim_H}(F)>1$. Then for almost all $\theta \in (-\frac{\pi}{2},\frac{\pi}{2})$, $\mathrm{proj_\theta}(F)$ (the projection of $F$ on the line $l_\theta:y=\mathrm{tan}\theta \cdot x$) has positive length.
 \end{fact}

 \begin{proposition}\label{36}
     If $A$ is an analytic subgroup of $(\mathbb{R},+)$ with $\mathrm{dim_H}(A)>\frac{1}{2}$, then $X_A$ is co-null.
 \end{proposition}

 This is a minor generalization of \cite[Exercise 6.8]{MR3236784}, which implies that for almost all $x$, $A+xA$ has Hausdorff dimension 1.

\begin{proof}
    By the product formula of Hausdorff dimension (see \cite[Chapter 7]{MR3236784}), $$\mathrm{dim_H}(A \times A) \ge\mathrm{dim_H}(A)+\mathrm{dim_H}(A)>1.$$
    Hence for almost all $\theta \in (-\frac{\pi}{2},\frac{\pi}{2})$, $\mathrm{proj_\theta}(A\times A)$ has positive length. Let $f(t)=\arctan t$ for all $t\in \mathbb{R}$. Define $\Theta=\{\theta \in (-\frac{\pi}{2},\frac{\pi}{2}):\mathrm{proj_\theta}(A\times A) \text{ has positive length}\}$. Since $\Theta$ is co-null (and hence measurable), so $f^{-1}(\Theta)$ is measurable. We claim that $f^{-1}(\Theta)$ is co-null. Suppose not, then $\mathbb{R} \backslash f^{-1}(\Theta)$ has positive measure. Then there is $N \in \mathbb{R}$ such that $\mathbb{R} \backslash f^{-1}(\Theta) \cap [-N,N]$ has positive measure. But $\inf_{t\in [-N,N]}f'(t)>0$, then $f(f^{-1}(\Theta) \cap [-N,N])$ has positive measure, a contradiction.
    
    Suppose $(a,b)$, $(c,d) \in A \times A$, $a \neq c$ and $b \neq d$. Then the line passing through these two points is:
    $$l_{(a,b),(c,d)}:\frac{y-d}{x-c}=\frac{y-b}{x-a}.$$
    The slope of $l_{(a,b),(c,d)}$ is $\frac{b-d}{a-c}$. And if $t\neq 0$ and $a+tb=c+td$, then $t=-\frac{a-c}{b-d}$. So for $(a,b),(c,d)$ and $t$ as above, $a+tb=c+td$ if and only if $l_{(a,b),(c,d)}$ is perpendicular to $l_{\arctan t}: y= t\cdot x$. In particular, $l_{\arctan t}$ and $A+tA$ are homeomorphic via:
    $$(a,b) \mapsto a+bt.$$
    And $$|a+bt|=(1+t^2)|a|=\sqrt{1+t^2}\sqrt{a^2+b^2}=\sqrt{1+t^2}|\overrightarrow{0,(a,b)}|.$$
    Hence the measure of $A+t A$ is equal to the measure of
    $\mathrm{proj_{\arctan t}}(A \times A)$ times $\sqrt{1+t^2}$ provided $t \neq 0$. By Proposition \ref{5}(i), $A+tA$ is a group. Then by Corollary \ref{steinhaus}, $A+tA= \mathbb{R}$ if $\mathrm{proj}_{\arctan t}(A\times A)$ has positive length. Since $f^{-1}(\Theta)$ is co-null, $\mathrm{proj}_{\arctan t}(A\times A)$ has positive length for almost all $t$.
\end{proof}

Next, we prove Theorem \ref{2}, which refines Proposition \ref{36}.
\begin{proof}[Proof of Theorem \ref{2}]
  As a convention, we consider a real $a$ and its binary expansion $m + 0. a_1 a_2 \cdots a_n \cdots$ with infinitely many 0s to be equivalent, where $m \in \mathbb{Z}$ is the integer part of $a$ and
  the numbers $a_i \in \{0, 1\}$ express its binary decimal part.

  We first proof a weaker result: there is an $F_\sigma$ subgroup $A$ of $(\mathbb{R},+)$ and $x \in \mathbb{R}$ such that $A + x A =\mathbb{R}$ and $A$ is null. 
  
  Define $Q_n =\{m \in \mathbb{N}: 3^{n - 1} \leqslant m <
  3^n \} $ for $n \ge 1$. Then each $Q_n$ contains $2 \cdot 3^{n - 1}$
  numbers. Divide $\{1, 2, . . .\}$ into two parts $P_1, P_2$ such that for all $n
  \ge 1$, the first $3^{n - 1}$ numbers of $Q_n$ are in $P_1$, and the
  other numbers of $Q_n$ are in $P_2$. Formally,
  \[ P_1 = \{ k : \exists n \exists i \in [0,3^n)  (k = 3^n + i) \}, \]
  \[ P_2 = \{ k : \exists n \exists i \in [3^n, 2 \cdot 3^n)  (k = 3^n + i) \} . \]
  Define
  \[ A_0 = \{ m + 0. a_1 a_2 \cdot \cdot \cdot a_n \cdots
      : m \in \mathbb{Z} \wedge \forall n \in P_2  (a_n =
     0) \} . \]
$A_0$ consists of reals with segments of 0s of length $3^k  (k \in \mathbb{N})$, and
  ahead of each segment of 0s there is an arbitrary segment of the same
  length. Define for each $n$, 
\[ A_n = \overbrace{A_0 \pm A_0 \pm \cdots \pm A_0}^n = \{ \alpha_1 \pm
     \alpha_2 \pm \cdots \pm \alpha_n : \alpha_i \in A_0, 1 \leqslant i
     \leqslant n \} . \]
Fix $n \ge 1$ and $\varepsilon >0$. Choose $k \ge 1$ such
  that 
  
  (i) $2^{- \frac{1}{2} \cdot 3^k + 1} < \varepsilon$; and
 
  (ii) for all $a = m +
  0. a_1 a_2 \cdots \in A_n$,
 
 $$\forall j \in \left[ 2 \cdot 3^{k -
  1}, \frac{5}{2} \cdot 3^{k - 1} \right](a_j=0) \lor \forall j \in \left[ 2 \cdot 3^{k -
  1}, \frac{5}{2} \cdot 3^{k - 1} \right](a_j=1).$$
  Then the measure of $A_n$ is no more than $2 \cdot 2^{- \frac{1}{2} \cdot
  3^k} < \varepsilon$. Since $\varepsilon$ is arbitrary, $A_n$ is null. Let $A$ be the subgroup of $(\mathbb{R},+)$ generated from $A_0$. Then $A = \bigcup_{n
  \ge 1} A_n$ is null. Since $A_0$ is compact, $A$ is $F_{\sigma}$. Define 
  $x = 0. x_1 x_2 \cdots$ such that $x_n = 1$ if and only if $\exists k(n = 3^k)$. Given any real $y = m + 0. y_1 y_2 \cdots$, we shall define a real $b = 0.
  b_1 b_2 \cdots \in A_0$ such that $b \cdot x$ and $y$ are equal on the
  $n$th binary place for each $n \in P_2$. Then we can pick $c
  \in A_0$ such that $c + b x = y$, so $A + x A =\mathbb{R}$.
  
  Intuitively, in the calculation $b \cdot x$, $x_{3^k} = 1$ causes the
  binary places of $b$ to shift to the right $3^k$ places, i.e.
  \[ 0. \overbrace{0 \cdots 0}^{3^k - 1} 1 \times 0. b_1 b_2 \cdots = 0.
     \overbrace{0 \cdots 0}^{3^k} b_1 b_2 \cdots . \]
  Denote $b x = 0. c_1 c_2 \cdots .$ By recursion on $n$, one can choose $\{b_m
  : m \in Q_n \cap P_1 \}$ such that $c_m = y_m$ for all $m \in Q_n \cap P_2$
  as follows: suppose $\{b_m : m \in Q_k \cap P_1, k < n\}$ have been chosen so that
  $c_m = y_m$ for all $k < n$ and $m \in Q_k \cap P_2$. Whatever $\{b_m :
  m \in Q_k \cap P_1, k \ge n + 1\}$ are chosen, we have
  \[ 0. \overbrace{0 \cdots 0}^{3^n - 1} b_{3^n} b_{3^n + 1} b_{3^n + 2}
     \cdots \times x < 0. \overbrace{0 \cdots 0}^{3^n-1} 1. \]
  Hence $\{c_m : m \in Q_n \cap P_2 \}$ is determined only by $\{b_m : m \in
  Q_k \cap P_1, k \leqslant n\}$. Now it is clear that we can choose $\{b_m : m \in Q_n \cap P_1 \}$ to
  ensure $c_m = y_m$ for all $m \in Q_n \cap P_2$. This ends the proof of the weaker result.

  It needs some more effort to make $\mathrm{dim_H}  (A) = \frac{1}{2}$. Let $\{p_n\}_{n \ge 1}$ be an increasing sequence of even numbers such that $\sum_{k<n}p_k=o(p_n)$ and $3^n=o(p_n)$. Define
  $$Q_n=\{m \in \mathbb{N}:1+ \sum_{k<n}p_k \le m <1+ \sum_{k \le n}p_k\}.$$
  Divide $\{1, 2, . . .\}$ into two parts $P_1, P_2$ such that for all $n
  \ge 1$, the first half of the numbers of $Q_n$ are in $P_1$, and the
  other numbers of $Q_n$ are in $P_2$. Define
  \[ A_0 = \{ m + 0. a_1 a_2 \cdot \cdot \cdot a_n \cdots
      : m \in \mathbb{Z} \wedge \forall n \in P_2  (a_n =
     0) \} . \]
 Define $A_n = \overbrace{A_0 \pm A_0 \pm \cdots \pm A_0}^n$ and $A = \bigcup_{n
  \ge 1} A_n$. As before, $A$ is an $F_\sigma$ subgroup of $(\mathbb{R},+)$ and there is $x$ such that $A+x A=\mathbb{R}$. It remains to show that $\mathrm{dim_H}(A_n)=\frac{1}{2}$ for all $n$. Fix $n$. For all $k>n$ and $a=m+0.a_1a_2 \cdots \in A_n$,
  $$\forall j \in \left( \sum_{j<k}p_j+\frac{p_k}{2}, \sum_{j \le k}p_j-n \right](a_j=0) \lor \forall j \in \left( \sum_{j<k}p_j+\frac{p_k}{2}, \sum_{j \le k}p_j-n \right](a_j=1).$$
Let $\delta=2^{-\sum_{j \le k}p_j}$. If $s >\frac{1}{2}$, then 
$$ \mathcal{H}^s_\delta(A_n) \le 2^{1+\sum_{j<k}p_j+\frac{p_k}{2}}2^{n+3}(2^{-\sum_{j \le k}p_j})^s \le 2^{n+4+\sum_{j<k}p_j-(s-\frac{1}{2})p_k}.$$
  Since $\sum_{j<k}p_j=o(p_k)$, $\mathcal{H}^s(A_n)=0$. So $\mathrm{dim_H}(A_n) \le \frac{1}{2}$. \end{proof}

\begin{remark}
   Note that the analogue of the above holds in the $p$-adics as well. Namely, there is an $F_\sigma$ subgroup $A\subseteq (\mathbb{Z}_p, +)$ and $x\in \mathbb{Z}_p$ such that $\mu(A)=0$ and $A+xA=\mathbb{Z}_p$, where $\mu$ denotes the Haar measure on $\mathbb{Z}_p$. Moreover, one can also achieve that the Hausdorff dimension is $1/2$. The proof follows from the same strategy as the previous proof.  We include the details for the sake of completeness.
\end{remark}
\begin{proof}
We only include the proof for the existence of such $A$, the calculation of the Hausdorff dimension of $A$ is left to the reader as it follows from the same argument as in $\mathbb{R}$.

Note that every element $a \in \mathbb Z_p$ has a unique expansion
\[
a=\sum_{n=0}^\infty a_n p^n,
\qquad a_n\in\{0,\dots,p-1\}.
\]
    Define
\[
Q_n=\{m\in\mathbb N:p^{n-1}\le m<p^n\}.
\]
Then
\[
|Q_n|=(p-1)p^{n-1}.
\]

Partition $\mathbb N$ into two sets $P_1,P_2$ so that for every
$n\ge1$, the first half of the indices in $Q_n$ belong to $P_1$
and the remaining indices belong to $P_2$.
Formally,
\[
P_1\cap Q_n
=
\left\{
p^{n-1},
\dots,
p^{n-1}+\Bigl\lfloor\frac{|Q_n|}{2}\Bigr\rfloor-1
\right\},
\]
and
\[
P_2=\mathbb N\setminus P_1.
\]

Define
\[
A_0=
\left\{
\sum_{n=0}^\infty a_n p^n:
\forall n\in P_2\ (a_n=0)
\right\}.
\]

Thus elements of $A_0$ have arbitrary digits on $P_1$ and vanish on
$P_2$.

For each $m\ge1$, define
\[
A_m=
\underbrace{A_0\pm A_0\pm\cdots\pm A_0}_{m}.
\]

Let
\[
A=\bigcup_{m\ge1}A_m.
\]

Since $A_0$ is compact, every $A_m$ is compact, hence $A$ is an
$F_\sigma$ subgroup of $(\mathbb Z_p,+)$.

Fix $m\ge1$.
Carrying in $p$-adic addition propagate only finitely far, which is comparable to $m$.
Hence there exists $N=N(m)$ such that for every sufficiently large
$k$ and every
\[
a=\sum a_n p^n\in A_m,
\]
either
\[
\forall j\in I_k\ (a_j=0)
\]
or
\[
\forall j\in I_k\ (a_j=p-1),
\]
where $I_k$ is a terminal segment of $Q_k$ of length comparable to
$|Q_k|/2$.

Therefore $A_m$ can be covered by at most
\[
2p^{|Q_1|+\cdots+|Q_{k-1}|+|Q_k|/2+O(1)}
\]
balls of radius
\[
p^{-\max Q_k}.
\]

Since a ball of radius $p^{-N}$ has Haar measure $p^{-N}$, it follows
that
\[
\mu(A_m)\le Cp^{-c|Q_k|}
\to0.
\]

Hence every $A_m$ is Haar null, and therefore $A$ is Haar null.

Now define
\[
x=\sum_{k=1}^\infty p^{p^k}.
\]

Thus the $p^k$-th digit of $x$ equals $1$, and all other digits vanish.

Let
\[
y=\sum_{n=0}^\infty y_n p^n\in\mathbb Z_p.
\]

We recursively construct
\[
b=\sum b_n p^n\in A_0
\]
such that
\[
bx
\]
agrees with $y$ on all coordinates in $P_2$.

Multiplication by $p^{p^k}$ shifts digits exactly $p^k$ places:
\[
p^{p^k}\sum b_n p^n
=
\sum b_n p^{n+p^k}.
\]

Because the supports are widely separated, digits coming from later
blocks cannot influence the current block except through bounded carry
effects. Hence, recursively on $k$, one may choose the free digits
\[
\{b_n:n\in P_1\cap Q_k\}
\]
so that
\[
(bx)_n=y_n
\qquad
(n\in P_2\cap Q_k).
\]

Define
\[
c=y-bx.
\]

Since $bx$ already agrees with $y$ on $P_2$, the element $c$ has
vanishing digits on $P_2$, hence
\[
c\in A_0.
\]

Therefore
\[
y=c+bx\in A+xA.
\]

Since $y\in\mathbb Z_p$ was arbitrary,
\[
A+xA=\mathbb Z_p.
\]
\end{proof}
In the proof of Theorem \ref{2}, both $A$ and $x$ depend on the intervals $Q_n$. What is the relation between $A$ and elements of $X_A$ in general?

\begin{proposition}
    If  $A$ is an analytic subgroup of $(\mathbb{R},+)$ and $x \in X_A$, then there are $c, d \in A$ such that $x=\frac{c}{d}$. In particular, $cA+dA=\mathbb{R}$.
\end{proposition}

This result is an immediate corollary of the following Theorem of Le Gac.

\begin{fact}[Barth\'elemy Le Gac \cite{MR687640}]
If $G$ and $H$ are analytic subgroups of $(\mathbb{R},+)$ such that $G+H=\mathbb{R}$ and $G \cap H =\{0\}$, then either $G=\mathbb{R}$ or $G=\{0\}$.
\end{fact}

Recall that for each $\alpha \in [0,1]$, there is a Borel subgroup $A_\alpha$ of $(\mathbb{R},+)$ with $\mathrm{dim_H}(A)=\alpha$ \cite{1}. For $\alpha >\frac{1}{2}$, $X_{A_\alpha}$ is co-null by Proposition \ref{36}. The following theorem claims that for such groups $A_\alpha$, there is an $F_\sigma$ subgroup $B\subseteq A_\alpha$ such that $X_B$ is co-null. It also reduces the problem of constructing an $F_{\sigma}$ proper subgroup $B$ of the reals with $X_B \neq \emptyset$ to the problem of constructing an analytic proper subgroup $A$ of the reals with $X_A \neq \emptyset$.

\begin{theorem}\label{fsigma}
 Suppose that $A$ is an analytic subgroup of $(\mathbb{R},+)$.

(i) If $x \in X_A$, then there is an $F_{\sigma}$ subgroup $B\subseteq A$ such that $x \in X_B$. 

(ii) If $X_A$ is co-null, then there is an $F_{\sigma}$  subgroup $B\subseteq A$ such that $X_B$ is co-null.
\end{theorem}

By Proposition \ref{5}(vi), $X_A=\bigcup_{q\in \mathbb{Q}}q+X_A$. Thus if $X_A$ has positive measure, it is automatically co-null. We need the following uniformization theorem in the proof.
\begin{fact}[Jankov, von Neumann, see {\cite[18.1]{00722611}}]\label{lemma: uniformizing} Let $X$, $Y$ be standard Borel spaces and $P \subseteq X \times Y$ is analytic. Then there is a $\Sigma$-measurable function $f:X \to Y$ such that $(x,f(x)) \in P$ for all $x \in proj_X(P)$ (the projection from $P$ to $X$), where $\Sigma$ is the $\sigma$-algebra generated by the analytic sets.
\end{fact}

In particular, if $X$ and $Y$ are Borel subsets of Euclidean spaces, then the function $f$ is Lebesgue measurable by Fact~\ref{Jech}.

\begin{proof}[Proof of Theorem \ref{fsigma}.]
For (i), suppose that $A$ is an analytic subgroup of $(\mathbb{R},+)$ and $x \in X_A$. Then the graph of the function $\varphi:A\times A \to \mathbb{R}$ so that $(a,b)\mapsto a+xb$ is an analytic subset of $\mathbb{R}^3$. By Fact \ref{lemma: uniformizing}, there is a Lebesgue measurable function $f:\mathbb{R} \to A\times A$ such that for any $z \in \mathbb{R}$, $\varphi(f(z))=z$. By the Lusin Theorem, there is a compact set $P\subseteq \mathbb{R}$ with positive measure and a continuous function $\psi: P\to A\times A$ so that for any $z\in P$, $\varphi(\psi(z))=z$. Let 
$$C_0=\{a \in \mathbb{R}:\exists z \in P \exists b \in \mathbb{R}(\psi(z)=(a,b))\},$$
and 
$$C_1=\{b \in \mathbb{R}:\exists z \in P \exists a \in \mathbb{R}(\psi(z)=(a,b))\}.$$
Both $C_0$ and $C_1$ are compact and so is $C_0\cup C_1$. Let $B$ be the group generated by $C_0\cup C_1 $.  Then $B$ is $F_{\sigma}$ and $B \subseteq A$. Moreover, the image $\varphi(B \times B)\supseteq P$ has positive measure. Since $\varphi(B \times B)$ is a subgroup of $\mathbb{R}$, by Corollary \ref{steinhaus}, $\varphi(B \times B)=\mathbb{R}$. Hence $x \in X_B$.

For (ii), suppose that $A$ is an analytic subgroup of $(\mathbb{R},+)$ and $X_A$ is co-null. By the regularity of Lebesgue measure, there is a co-null $F_{\sigma}$ subset $F \subseteq X_A$. It suffices to
show that there is an $F_\sigma$ subgroup $B \subseteq A$ such that for almost all $x \in
F$, we have $x\in X_B$. 

The graph of the function
\[ \varphi : A \times A \times F \rightarrow \mathbb{R} \times F, (a, b, x)
   \mapsto (a + x b, x) \]
is an analytic subset of $\mathbb{R}^5$. By Fact \ref{lemma: uniformizing}, there is a Lebesgue
measurable function $$f : \mathbb{R} \times F \rightarrow A \times A \times F$$
such that for any $z \in \mathbb{R} \times F$, $\varphi (f (z)) = z$. By
the Lusin Theorem, for all $n$, there is a compact subset $K_n$ of $\mathbb{R}
\times F$ such that $\psi_n := f \upharpoonright K_n$ is continuous and
$$\mu ((\mathbb{R} \times F) \cap B_n (0) \backslash K_n) < \frac{1}{n},$$
where $\mu$ is the Lebesgue measure on $\mathbb{R}^2$ and $B_n (0) \subseteq
\mathbb{R}^2$ is the closed ball with radius $n$ and center $0$. Let 
$$C_0=\{a \in \mathbb{R}:\exists n\exists z \in K_n \exists b,x \in \mathbb{R}(\psi_n(z)=(a,b,x))\},$$
and 
$$C_1=\{b \in \mathbb{R}:\exists n\exists z \in K_n \exists a,x \in \mathbb{R}(\psi(z)=(a,b,x))\}.$$
Both $C_0$ and $C_1$ are $F_\sigma$ and so is $C_0\cup C_1$. Let $B$ be the additive group generated by $C_0\cup C_1 $.  Then $B$ is an $F_{\sigma}$ set and $B \subseteq A$. Moreover, the projection of the image $\varphi(B \times B \times \{x\})$ to the first coordinate has positive measure for almost all $x \in F$ by the Fubini theorem since $(\mathbb{R}\times F) \backslash \cup_n K_n$ is null. So $X_B$ is conull.
\end{proof}

\section{Constructions assuming CH}\label{s4}

The main goal of this section is to prove Theorem \ref{main}. It requires notion from recursion theory and algorithmic dimension. A standard reference is \cite{MR2732288}. We consider the
elements of the Cantor space $2^{\omega}$ as reals. We denote the set of binary strings of finite
length by $2^{< \omega}$. Given $\sigma, \tau \in 2^{< \omega}$, we write
$\sigma \prec \tau$ if $\sigma$ is a proper initial segment of $\tau$. The
same notation is applied when $\tau$ is replaced by a real $x \in 2^{\omega}$.
We write $\sigma \tau$ to denote the string obtained by concatenating $\sigma$
and $\tau$. The Cantor space is equipped with a topology generated by the
basic clopen sets $I_{\sigma} = \{ \sigma \alpha : \alpha \in 2^{\omega} \}$
for $\sigma \in 2^{< \omega}$. It is also a measure space: the Lebesgue measure $\mu  (I_{\sigma}) = 2^{- |
\sigma |}$, where $| \sigma |$ is the length of the string $\sigma$. For a co-infinite set $z \subseteq \omega$, define 
$$F(z)=\sum_{i \in z}  2^{-i-1}
\in [0,1) \subseteq \mathbb{R}.$$
$F$ is an ``isometry'' between the co-null subset of $2^\omega$ consisting of the co-infinite sets and the interval $[0,1)$. Note that under $F$, the measure $\mu$ on $2^\omega$ turns into the Lebesgue measure on $\mathbb{R}$. Similarly the Hausdorff dimension is also preserved between $2^\omega$ and $\mathbb{R}$. For a rigorous proof of this fact, see \cite[Section 4]{marks2024}. There the authors proved that $F$ preserves the property of having positive Hausdorff measure. Since the set of all the co-finite elements in $2^\omega$ is countable, it is $\mu$-null. Hence we assume that every real $x \in 2^\omega$ we deal with is co-infinite and consider the arithmetic operations on $2^{\omega}$ to be the same as arithmetic operations on $\mathbb{R} / \mathbb{Z}$. Given reals $x$, $y$. We write $x \leq_T y$ if $x$ is Turing reducible to $y$. Given a real $y$, if $W \subseteq 2^{< \omega}$ is r.e. (recursively enumerable) in
$y$, then the set $U \subseteq 2^{\omega}$ of reals with an initial segment in
$W$ is called a $\Sigma^0_1 (y)$ \textit{set}. Given $\sigma \in 2^{< \omega}$ and $x \in 2^{\omega}$, let $K (\sigma)$ be the prefix-free Kolmogorov complexity of $\sigma$ and $K^x (\sigma)$ be the prefix-free Kolmogorov complexity of $\sigma$ relativized to $x$ (see {\cite{MR2732288}}). Given reals $x, y \in 2^{\omega}$, define the real $x \oplus y$ such that for
each $n$, $x \oplus y  (2 n) = x  (n)$ and $x \oplus y  (2 n + 1) = y  (n)$. Note that operation $\oplus$ is not associative. However, it is invariant under Turing degree. For example, we have $(x_1 \oplus x_2)\oplus x_3 \equiv_T x_1 \oplus (x_2\oplus x_3)$. Then we define by recursion that $$x_1\oplus\cdots \oplus x_n=x_1\oplus (x_2\oplus x_3\oplus \cdots \oplus x_n).$$ 

\begin{definition}
  (i) A set $S \subseteq 2^{< \omega}$ is \emph{dense} if for every $\sigma \in 2^{<
  \omega}$, there is a string $\tau \in S$ such that $\tau \succ \sigma$.
  
  (ii) Given reals $x$ and $y$. We say that $x$ is \emph{$y$-generic}\footnote{Our definition of $y$-generic reals should be
named as the \textit{weakly 1-$y$-generic reals} according to the tradition. Since there is no need for the more general concept of
$n$-genericity in this paper, we state the definition as above for
convenience. } if for every
  $\Sigma^0_1 (y)$ dense set $S \subseteq 2^{< \omega}$, there is a string
  $\sigma \prec x$ such that $\sigma \in S$. We say $x$ is \emph{generic} if
  $x$ is $x_0$-generic where $x_0(n)=0$ for all $n \in \omega$. In particular, if $x$ is $y$-generic for some $y$, then $x$ is generic.
\end{definition}

\begin{fact}[H{\"o}lzl et al. {\cite{MR4085071}}]
  \label{l9}There is a constant $c$ such that for every generic real $x$ and
  $i \in \{0, 1\}$, there are infinitely many $n$ such that
  \[  \forall m \in \left[ n, 2^{2^{2^n}} \right] (K (G
     \upharpoonright m) \leqslant K (m) + c \wedge G (m) = i) . \]
\end{fact}

\begin{fact}[The point-to-set principle, Lutz and
Lutz {\cite{MR3811993}}]\label{pts}
  For every set $E \subseteq \mathbb{R}$,
  \[ \mathrm{dim_H} (E) = \min_{A \subseteq \omega} \sup_{x \in E} \liminf_{n
     \rightarrow \infty} \frac{K^A (x \upharpoonright n)}{n} . \]
\end{fact}

\begin{lemma}[Folklore]
  \label{10}\label{09}Let $G = \{ x \in 2^{\omega} : x \text{ } \mathrm{generic} \}$, then $\mathrm{dim_H} (G) = 0$.
\end{lemma}

\begin{proof}
  Note that $K (n) \leqslant \log n + 2 \log \log n + O (1)$ for every $n \in
  \omega$ (see {\cite{MR2732288}}). Then by Fact \ref{l9}, we have
  $ \liminf_{n \rightarrow \infty} \frac{K (x \upharpoonright n)}{n} = 0 $
  for every generic real $x$. Hence by Fact \ref{pts},
  $\mathrm{dim_H} (G) = 0$.
\end{proof}

The following lemma reflects the intuition that genericity is preserved under
arithmetic operations. Recall that we consider the arithmetic operations on $2^{\omega}$ to be the same as arithmetic operations on $\mathbb{R} / \mathbb{Z}$. 

\begin{lemma}
  \label{8}Given $x, a, b, g \in 2^{\omega} \backslash \{0\}$, if $g$ is $a \oplus b \oplus x$-generic, then $g + b$, $a \cdot g$, $g^{n} (n \in \mathbb{Z}\backslash \{0\})$ and $a
  \cdot g + b$ are $a \oplus b \oplus x$-generic.
\end{lemma}

\begin{proof}
We proof that $g+b$ is $a \oplus b \oplus x$-generic first.
  Fix a $\Sigma_1^0
  (a \oplus b \oplus x)$ dense set $S$ and $\{S_s\}_{s \in \omega}$ an $a \oplus b \oplus x$-recursive enumeration of $S$ ($S_n \subseteq S_{n+1}$ for each $n$ and there is exactly one element in $S_{n+1} \backslash S_{n}$). We inductively define $S {}'_s$ at stage $s+1$ as follows: 
  
  suppose $\sigma \in S_{s+1} \backslash S_{s}$. Since there are infinitely many 0s in the sequence $b$, there is $\tau$ such that $\forall y \succ \tau (y + b \succ \sigma)$. Choose $\tau$ of the least length and enumerate it into $S{}'$.

  It is clear that $S{}'$ is $\Sigma_1^0
  (a \oplus b \oplus x)$. Fix $\gamma \in 2^{<\omega}$. Choose $\gamma_1$ such that $\{y+b:y \succ \gamma \} \supseteq I_{\gamma_1}$. Since $S$ is dense, there is $\rho \in S$ such that $\rho \succ \gamma_1$. By the definition of $S'$, there is $\tau \in S'$ such that $\forall y \succ \tau (y+b \succ \rho)$. Then $\tau \succeq \gamma$. Hence $S'$ is dense. Since $g$ is $a \oplus b \oplus x$-generic, there is $\tau \prec g$ such that $\tau \in S{}'$. By the definition of $S{}'$, there is $\sigma \in S$ such that $\sigma \prec g+b$. So $g+b$ is $a \oplus b \oplus x$-generic.

  Similarly one can proof that $a \cdot g$, $g^{-1}$ and $g^{2}$ are $a \oplus b \oplus x$-generic, which implies $g^{n}$ and $a
  \cdot g + b$ also are.
\end{proof}

\begin{proof}[Proof of Theorem \ref{main}]
  Fix $\{ (x_{\alpha}, y_{\alpha})\}_{\alpha < \aleph_1}$ an enumeration of
  $ (\mathrm{dom}(F)\backslash F^{-1}(\mathbb{Q})) \times \mathrm{dom}(F)$ ($F$ is defined at the beginning of this section). We construct a
  sequence of pairs of reals $\{ (g_{\alpha}, h_{\alpha})\}_{\alpha <
  \aleph_1}$ by induction on $\alpha < \aleph_1$.
  
  \bigskip
  \textbf{Stage $\alpha$.} Define 
  \[ G_{\alpha}=\{ x_{\beta} \oplus y_{\beta} \oplus g_{\beta} \oplus h_{\beta} \oplus
     x_{\alpha} \oplus y_{\alpha} : \beta < \alpha \} . \]
 Let 
  \[I_{\alpha}=\{x:(\exists n)(\exists a_0,...,a_n \in G_{\alpha})[x \leq_T a_0 \oplus a_1 \oplus ... \oplus a_n]\},\] 
 the Turing ideal generated by $G_{\alpha}$. Since $I_{\alpha}$ is countable, there is a real $g_{\alpha}$ that is
  $z$-generic for all $z \in I_{\alpha}$. Let $h_{\alpha} = \frac{y_{\alpha}
  - g_{\alpha}}{x_{\alpha}}$.
  
  \bigskip
  Let $A_0$ be the group generated by $\{g_{\alpha} : \alpha <
  \aleph_1 \} \cup \{h_{\alpha} : \alpha < \aleph_1 \}$. We claim that $A_0$ contains only generic reals. For any $g \in A_0$, there are finite sequences $\{g_{\alpha_i}, h_{\alpha_i},
  s_i, t_i \}_{0 \leqslant i \leqslant n}$, where $s_i$, $t_i \in \mathbb{Z}$, and ordinals
  $\alpha_0 < \alpha_1 < \cdots < \alpha_n$ such that $s_n^2+t_n^2\neq 0$ and
  \[ g = \sum_{i = 0}^n  (s_i g_{\alpha_i} + t_i h_{\alpha_i}) . \]
  Define $c =  (\oplus_{0 \leqslant i \leqslant n - 1}  (g_{\alpha_i}
  \oplus h_{\alpha_i})) \oplus x_{\alpha_n} \oplus y_{\alpha_n}$. By
  the construction, $g_{\alpha_n}$ is $c$-generic. Also
  \[ s_n g_{\alpha_n} + t_n h_{\alpha_n} = t_n
     \frac{y_{\alpha_n}}{x_{\alpha_n}} + \left ( s_n - \frac{t_n}{x_{\alpha_n}}
     \right) g_{\alpha_n}, \]
  so
  \[ g = \left ( s_n - \frac{t_n}{x_{\alpha_n}} \right) g_{\alpha_n} + \left (
     t_n \frac{y_{\alpha_n}}{x_{\alpha_n}} + \sum_{i = 0}^{n - 1}  (s_i
     g_{\alpha_i} + t_i h_{\alpha_i}) \right) . \]
  Since $x_{\alpha_n}$ is irrational, we have that $s_n -
  \frac{t_n}{x_{\alpha_n}} \neq 0$. By Lemma \ref{8}, $g$ is $c$-generic and
  hence generic. Thus $\mathrm{dim_H}  (A_0) = 0$ by Lemma \ref{09}. Let $A$ be the subgroup of $\mathbb{R}$ generated by $F(A_0)$. Then $\mathrm{dim_H}  (A) = 0$. For any pair $ (x, y) \in [0,1] \times [0,1]$ with $x \not\in \mathbb{Q}$, there are members $g, h \in
  A$ so that $h = \frac{y - g}{x}$ and so $g + x \cdot h = y$. So $\mathbb{R}\backslash \mathbb{Q} \subseteq X_A$. By Proposition \ref{5}(ii), $\mathbb{Q} \cap X_A= \emptyset$.
\end{proof}

\begin{question} (i) Can $\mathrm{CH}$ in Theorem \ref{main}  be
  removed?

  (ii) Can $A$ in Theorem \ref{main} be Borel? 
\end{question}

The complexity of $A$ in Theorem \ref{main} is not evident from the construction. The following proposition shows that at least it is impossible to be $F_\sigma$.

\begin{proposition}\label{48} Suppose $A \subseteq \mathbb{R}$ is $F_{\sigma}$ and null,\footnote{$A$ is not necessarily a group.} then there is an irrational $x$ such that $A+xA$ is null. Moreover, the set $\{x  :  A+xA \text{ } \mathrm{null}\}$ is comeager.\footnote{A \textit{comeager} set is the complement of a set of the first category.}
\end{proposition}

Let $\mu$ be the Lebesgue measure on $\mathbb{R}$. The following fact is clear.

\begin{fact}\label{lem: basic addition}
Let $\{I_i\}_{i\leq n}$ be a finite set of open intervals and $J$ an open interval, then $$\mu((\bigcup_{i\leq n}I_i)+J)\leq n\mu(J)+\sum_{k\leq n}\mu(I_i).$$
\end{fact}

\begin{proof}[Proof of Proposition \ref{48}]
Suppose $A = \bigcup_{n \ge 1} A_n$ is an $F_{\sigma}$ null subset of
$\mathbb{R}$ with each $A_n$ compact. 

\begin{lemma}\label{048}
    For each
$m, n \ge 1$, the set $$D_{m, n} = \left\{ x \in \mathbb{R}: \mu
(A_n + x A_n) < \frac{1}{m} \right\}$$ contains  a dense open subset of $\mathbb{R}$.
\end{lemma}

\begin{proof}[Proof of Lemma \ref{048}]
    Given $m, n \ge 1$, we shall show that $D_{m, n}$
contains a dense open set in $(0, 1)$. It is not hard to generalize the proof to show that
$D_{m, n}$ contains a dense open set in $\mathbb{R}$.

Let $\sigma$ be a binary string and $q \in \mathbb{N}$. Fix a
real $x = 0. \sigma 0^q x^{\ast} \in (0, 1)$, where $x^{\ast} \in 2^{\omega}$ is the tail of the binary expansion of $x$. Define $y = 0. x^{\ast} \in (0, 1)$. Then \[ A_n + x A_n \subseteq A_n + \sum_{0 \leqslant i \leqslant | \sigma | - 1} \sigma
   (i) 2^{- i - 1} A_n + 2^{- | \sigma | - q} y A_n   . \]
   The set $A_{n,\sigma}:=A_n+\sum_{0 \leqslant i \leqslant | \sigma | - 1} \sigma
   (i) 2^{- i - 1} A_n$ is compact and 
$$A_{n,\sigma}\subseteq A+\sum_{0 \leqslant i \leqslant | \sigma | - 1} \sigma(i) 2^{- i - 1} A$$ is null by Proposition \ref{5}(i) and Corollary \ref{steinhaus}. So there is a finite open interval cover $\{ I_j \}_{1
\leqslant j \leqslant k}$ of $A_{n,\sigma}$ such that
$\sum_{1 \leqslant j \leqslant k} | I_j | < \frac{1}{2 m}$. Suppose that $A_n \subseteq [- N, N]$ and so  $yA_n \subseteq [- N, N]$. Then if
$2^{- q+1}  Nk < \frac{1}{2m}$, by Lemma \ref{lem: basic addition},  we have
$$\mu (A_n + x A_n) \le\mu (   A_{n,\sigma} + 2^{- | \sigma | - q} y A_n )   \le \sum_{0\leq j\leq k}| I_j|+k 2^{- q-|\sigma|}  2N   < \frac{1}{2m}+\frac{1}{2m}=\frac{1}{m}.$$
Since $\sigma$ is arbitrary, $D_{m,n}$ contains a dense open set in $(0,1)$.
\end{proof}

Using this lemma, the set $\{x  :  \mu(A+xA)=0\}$ is comeager by the Baire category theorem.\end{proof}

Assuming $\mathrm{CH}$, we can also construct a maximal subfield of $\mathbb{R}$ with Hausdorff dimension 0 such that some given real is not in its algebraic closure. Given a set $A \subseteq{\mathbb{R}}$, let $F(A)$ be the field generated by $A$. Denote the relative algebraic closure of $F(A)$ in $\mathbb{R}$ by $\mathrm{acl}_{\mathbb{R}}(A)$.

\begin{fact}[Exchange of Algebraic Closure]\label{047} For a set $A \subseteq \mathbb{R}$ and reals $a$ and $b$, if $b \in \mathrm{acl}_{\mathbb{R}}(A \cup \{a\})$ and $b \notin \mathrm{acl}_{\mathbb{R}}(A)$, then $a \in \mathrm{acl}_{\mathbb{R}}(A \cup \{b\})$.

\end{fact}

\begin{proposition}\label{47}
Assume $\mathrm{CH}$. Given $x$ a transcendental number, there is a subfield $A$ of $\mathbb{R}$
  such that
 
 (i) $\mathrm{dim_H} (A)=0$; and
 
 (ii) $x \notin \mathrm{acl}_{\mathbb{R}}(A)$; and
 
 (iii) for any $y \notin \mathrm{acl}_{\mathbb{R}}(A)$, $x \in \mathrm{acl}_{\mathbb{R}}(A \cup \{y\})$.
\end{proposition}
\begin{proof}
  Fix $\{ y_{\alpha} \}_{\alpha < \aleph_1}$ an enumeration of $\mathbb{R}$.
  Fix $x$ a transcendental number. Define $A = \bigcup_{\alpha < \aleph_1} A_{\alpha}$
  by induction on stages $\alpha < \aleph_1$:
  
  \
  
  \textbf{Stage $0$.} Define $A_0=\mathbb{Q}$. $x \notin \mathrm{acl}_{\mathbb{R}}(A_0)$ since it is transcendental.
  
  \textbf{Stage $\alpha>0$.} Recall that $G$ is the set of generic reals (see Lemma \ref{10}). Suppose by induction that we have a countable field $$B_{\alpha} =
  \bigcup_{\beta < \alpha} A_{\beta} \subseteq G \cup \mathbb{Q}$$ such that $x \notin \mathrm{acl}_{\mathbb{R}}(B_\alpha)$. Choose the least $\gamma < \aleph_1$ such that $y_{\gamma} \notin \mathrm{acl}_{\mathbb{R}}(B_{\alpha})$ and $x \notin \mathrm{acl}_{\mathbb{R}}(B_{\alpha} \cup \{y_{\gamma}\})$ (if such a $\gamma$ does not exist, end the whole construction). And if such a $\gamma$ exists, we say that $\gamma$ \textit{acts} at stage $\alpha$.
Let $I_\alpha$ be the Turing ideal generated by $B_\alpha \cup \{y_\gamma,x\}$.
 Choose a real $g_\alpha$ such that 
 
 (i) $g_\alpha$ is $y$-generic for all $y \in I_\alpha$; and 

 (ii) $g_\alpha \notin \mathrm{acl}_{\mathbb{R}}(B_\alpha \cup\{y_\gamma,x\} )$.
 
 Let $A_\alpha$ be the subfield of $\mathbb{R}$ generated by
$B_{\alpha} \cup \{ g_\alpha, \frac{g_\alpha-x}{y_\gamma}\}$. 

\

We first show that $x \notin \mathrm{acl}_{\mathbb{R}}(A_\alpha)$ if there is an ordinal $\gamma$ which acts at stage $\alpha$. By induction, we assume that $x \notin \mathrm{acl}_{\mathbb{R}}(B_\alpha)$. Let $h_\alpha=\frac{g_\alpha-x}{y_\gamma}$. Then $x\in \mathrm{acl}_{\mathbb{R}}(B_\alpha \cup \{g _\alpha ,h_\alpha\})$ but $x \notin \mathrm{acl}_{\mathbb{R}}(B_\alpha \cup \{g _\alpha \})$ ($x \notin \mathrm{acl}_{\mathbb{R}}(B_\alpha)$ by the induction hypothesis, and $g_\alpha \notin \mathrm{acl}_{\mathbb{R}}(B_\alpha \cup\{x\} )$ by its choice, then by Fact \ref{047}, $x \notin \mathrm{acl}_{\mathbb{R}}(B_\alpha \cup \{g _\alpha \})$). So $h_\alpha \in \mathrm{acl}_{\mathbb{R}}(B_\alpha \cup \{g _\alpha,x \})$ by Fact \ref{047}. Then $y_\gamma \in \mathrm{acl}_{\mathbb{R}}(B_\alpha \cup \{g _\alpha,x \})$. And by Fact \ref{047} again, since $y_\gamma \notin \mathrm{acl}_{\mathbb{R}}(B_\alpha)$, we have $y_\gamma \notin \mathrm{acl}_{\mathbb{R}}(B_\alpha \cup \{x\})$. So by Fact \ref{047} again, $g_\alpha \in \mathrm{acl}_{\mathbb{R}}(B_\alpha \cup \{y_\gamma,x\})$, a contradiction.

By Lemma \ref{8}, $A_{\alpha} \subseteq G \cup \mathbb{Q}$ for all $\alpha$. So the induction hypothesis holds at any stage before the construction ends.

Now we define $A = \bigcup_{\alpha < \aleph_1} A_{\alpha}$ and verify that $A$ satisfies the conditions in the statement. We may assume the construction does not end at any stage $\alpha <\aleph _1$ (otherwise $A$ is countable and it is clear that the conditions hold). Since $A = \bigcup_{\alpha < \aleph_1} A_{\alpha} \subseteq G \cup \mathbb{Q}$, $\mathrm{Dim_H} (A)=0$. If $y\notin \mathrm{acl}_{\mathbb{R}}(A)$, choose $\gamma < \aleph_1$ such that $y=y_\gamma$. Then $x \in \mathrm{acl}_{\mathbb{R}}(A\cup \{y\})$. Otherwise, $\gamma$ acts at some stage $\alpha < \aleph_1$. So $g_\alpha$, $\frac{g_\alpha-x}{y}$ and $y$ are in $\mathrm{acl}_{\mathbb{R}}(A\cup \{y\})$. Then $x$ is also in $\mathrm{acl}_{\mathbb{R}}(A\cup \{y\})$, a contradiction.
\end{proof}

  \bigskip

\end{document}